\documentclass{amsart}
\usepackage[hyperref]{waynemath}
\usepackage{waynealgtop}
\usepackage[all, cmtip]{xy}
\usepackage{caption}

\makeatletter
	\let\c@equation\c@theorem
\makeatother
\numberwithin{equation}{section}

\makeatletter
	\renewcommand\subsection{\@startsection{subsection}{2}%
		\z@{.5\linespacing\@plus.7\linespacing}{.5\linespacing}%
		{\normalfont\bfseries}} 
	\newcommand{\myitem}[1]{%
		\item[#1]\protected@edef\@currentlabel{#1}%
		}
\makeatother

\newcommand{\makebibliography}{
	\newpage
	\bibliographystyle{plain}
	\bibliography{Groebner.bib}
}
 
\theoremstyle{definition}

\newcommand{\Groebner}{Gr\"obner}
\newcommand{\FComm}{filtered-commutative}
\newcommand{\RN}[1]{\textup{\uppercase\expandafter{\romannumeral#1}}}

\newcommand{\gr}{\mathrm{gr}}
\newcommand{\pr}{\mathrm{pr}}
\newcommand{\LM}{\mathrm{LM}}
\newcommand{\Syz}{\mathrm{Syz}}
\newcommand{\LMg}{\LM_{\gr}}
\newcommand{\LMA}{\LM_{A}}
\newcommand{\LMB}{\LM_{\!\sB}}
\newcommand{\LMBB}{\LM_{\bB}}
\newcommand{\LT}{\mathrm{LT}}
\newcommand{\LTg}{\LT_{\gr}}
\newcommand{\LTA}{\LT_{A}}
\newcommand*{\pnul}{{p_{\mathrm{nul}}}}

\newcommand\blfootnote[1]{%
  \begingroup
  \renewcommand\thefootnote{}\footnote{#1}%
  \addtocounter{footnote}{-1}%
  \endgroup
}

\title[Noncommutative Gr\"obner Bases and Ext groups]{Noncommutative Gr\"obner Bases and Ext groups; Application to the Steenrod Algebra}
\author{Weinan Lin}
\date{}

\begin{document}
\begin{abstract}
	We consider a theory of noncommutative \Groebner{} bases on decreasingly filtered algebras whose associated graded algebras are commutative. We transfer many algorithms that use commutative \Groebner{} bases to this context. As an important application, we implement very efficient algorithms to compute the Ext groups over the Steenrod algebra $\sA$ at the prime $2$. Especially, the cohomology of the Steenrod algebra $\Ext_\sA^{*, *}(\bF_2, \bF_2)$, which plays an important role in algebraic topology, is calculated up to total degree of 261, including the ring structure in this range.
\end{abstract}
\maketitle
\tableofcontents
\blfootnote{Project funded by China Postdoctoral Science Foundation, 2021TQ0015}

\section{Introduction}
The commutative \Groebner{} basis theory \cite{Buchberger} has been a very powerful tool in different areas of mathematics such as commutative algebra, algebraic geometry, homological algebra and applied mathematics. There are several works on generalizing the \Groebner{} bases to noncommutative algebras, including \cite{AL98} for enveloping algebras of Lie algebras, \cite{Bergman78} and \cite{Mora94} for two-sided ideals in noncommutative free algebras and \cite{Li02} exploring the relationship between commutative and noncommutative \Groebner{} bases.

However, most of the \Groebner{} basis theories rely on restrictive monomial orderings that require $M\ge N$ when a monomial $M$ is a multiple of another monomial $N$. This paper introduces a different type of monomial orderings (Definition \ref{def:de6997ec}) which allows us to apply the theory of \Groebner{} bases to broader noncommutative algebras including the Steenrod algebra more effectively.

We define a noncommutative \Groebner{} basis theory on a class of algebras called \FComm{} algebras (Definition \ref{def:63d8c34e}). The main results are the following.

\begin{theorem}[Algorithm \ref{alg:30f8cc72}]
	Given a finite generating set $H$ ($0\notin H$) of a left ideal $J$ of a \FComm{} algebra $A$, there is a generalized Buchberger's algorithm that expands $H$ to a \Groebner{} basis of $J$.
\end{theorem}

\begin{theorem}[Algorithm \ref{alg:353e422a}]
	Let $N=A^r/M$ be a graded left module over a \FComm{} graded algebra $A$ where $M$ is a left submodule of $A^r$. There is an algorithm that constructs the first $s+1$ terms $F_0,\dots,F_s$ of a free $A$-resolution
	$$\cdots\to F_s\fto{d_s}\cdots\fto{d_3} F_2\fto{d_2} F_1\fto{d_1} F_0=A^r\fto{\epsilon} N$$
	for any $s$.
\end{theorem}

The algorithms above have many applications and we focus on developing an algorithm for computing $\Ext$ groups which is important in homological algebra and algebraic topology. We then apply the algorithm to the mod 2 Steenrod algebra $\sA$ which yields the following computational result.

\begin{theorem}\label{thm:f088d56}
	The bigraded algebra
	$$\bigoplus_{t\le 261}\Ext_{\sA}^{s,t}(\bF_2,\bF_2)$$
	in total degrees up to 261 is an algebra with 2914 indecomposables, 23822 basis elements and 227498 indecomposable relations. The complete charts are in given in \cite{LinZenodo}.
\end{theorem}

This algebra is the $E_2$ page of the Adams spectral sequence converging to the stable homotopy groups of spheres completed at the prime 2 (See \cite{Adams58}). It plays a fundamental role in the stable homotopy theory and is related to many important problems such as the Hopf invariant one problem and the Kervaire invariant one problem. The computation of the stable homotopy groups of spheres relies on these machine outputs for the Adams $E_2$ page. We refer the interested readers to Isaksen, Wang and Xu [6] for the most up to date calculation.

Previously, the most extensive and publicly available calculation of $\Ext_{\sA}^{*,*}(\bF_2,\bF_2)$ is done by Bruner and Rognes \cite{BR}, who compute the Ext up to dimension of 200. There are other computer programs made by Nassau \cite{Nassau} and Wang \cite{Wang}. Our new algorithm is very efficient in both memory and time, which can compute the resolution in the range of 200 within one day on a personal computer and we spent 88+63 days to obtain the resolution and products for Theorem \ref{thm:f088d56} on a 64-core CPU machine. This extends their calculation by a few magnitudes. The interested reader can find the latest implementation of the algorithm on Github \cite{LinGH}, which can run twice as fast as the old version does for Theorem \ref{thm:f088d56}.

\section{Filtered-Commutative Algebras}
Let $k$ be any field and $A$ an algebra defined below.
\begin{definition}\label{def:63d8c34e}
	An algebra $A$ over $k$ is called a \emph{(decreasingly-) \FComm{} algebra} if it satisfies the following conditions.
	\begin{enumerate}
		\item The algebra $A$ is equipped with a filtration of ideals
		$$A=F_0A\supset F_1A\supset F_2A\supset \cdots$$
		such that 
		$$F_pA\cdot F_qA\subset F_{p+q}A$$
		and $F_{\pnul}A=0$ for some number $\pnul>0$.
		\item The associated graded algebra $\gr(A)$ defined by
		$$\gr_p(A)=F_pA/F_{p+1}A$$
		is commutative.
	\end{enumerate}
\end{definition}

\begin{definition}
	We say that a nonzero element $a\in A$ projects to $f\in \gr_p(A)$, if $p$ is the maximal number such that $a\in F_pA$ and $f$ is the image of $a$ via the map $F_pA\to F_pA/F_{p+1}A$. We write 
	$$v(a)=p\text{ and }\pr(a)=f.$$
	If $a$ is zero, we define 
	$$v(0)=\infty\text{ and }\pr(0)=0.$$ 
	We also call $a$ a \emph{lift} of $f$.
\end{definition}

It is easy to see that another element $a^\prime\in A$ projects to the same element $f\in \gr_p(A)$ if and only if $a-a^\prime\in F_{p+1}A$. The projection map is \emph{not} a linear map between vector spaces since $\pr(a-a^\prime)$ here could be nonzero.

\begin{proposition}
	For nonzero elements $a,b\in A$, we have
	$$v(ab)\ge v(a)+v(b)$$
	or equivalently,
	$$|\pr(ab)|\ge |\pr(a)|+|\pr(b)|$$
	where $|\cdot|$ is the degree of elements of $\gr(A)$. The equality holds if and only if 
	$$\pr(a)\pr(b)\neq 0$$ 
	in $\gr(A)$.
\end{proposition}
\begin{proof}
	This is the direct consequence of the condition $F_pA\cdot F_qA\subset F_{p+q}A$.
\end{proof}

The condition $F_{\pnul}A=0$ in Definition \ref{def:63d8c34e} makes sure that a basis of $\gr(A)$ can lift to a basis of $A$. More precisely, if $\sB=\{m_i\}$ is a $k$-basis of $\gr(A)$, and $\tilde m_i\in A$ projects to $m_i$, then $\tilde \sB=\{\tilde m_i\}$ is a $k$-basis of $A$.

\begin{remark}
	The definition of a \FComm{} algebra is a generalization of commutative algebras since if $A$ is commutative, we can filter it with
	$$A=F_0A\subset F_1A=\{0\}.$$
	It is obvious that $A$ satisfies the conditions in Definition \ref{def:63d8c34e}. There are also non-commutative examples. The important ones that we will study in this paper are the truncated Steenrod algebra at the prime 2.
\end{remark}

From now on, we study \FComm{} algebras $A$ such that $\gr(A)$ is finitely generated. Hence the associated graded algebra can be presented by
\begin{equation*}
	\gr(A)=k[x_1,\dots,x_n]/I
\end{equation*}
for some $n$ and some ideal $I$ of the polynomial ring $P=k[x_1,\dots,x_n]$. Since $\gr(A)$ is graded, we assume that $x_i$ and $I$ are homogeneous. The degree of a homogeneous element $f\in P$ is denoted by $|f|$.

\section{Bases and Monomial Orderings}
In order to develop a theory of \Groebner{} bases over $A$, we want to choose a basis for $A$ and a well-order on the basis. Since $\gr(A)$ is a commutative algebra, we will do it on $\gr(A)$ first and then lift it to $A$.

\begin{definition}\label{def:de6997ec}
	A \emph{degree-reversed admissible monomial ordering} for the graded polynomial ring $P=k[x_1,\dots,x_n]$ is a total ordering on the monomials such that for any monomials $M, N, L$,
	\begin{enumerate}
		\item $M\le N\Longleftrightarrow ML\le NL$,
		\item $|L|=0\Longrightarrow M\le ML$.
		\item $|M|>|N|\Longrightarrow M<N$.
	\end{enumerate}
\end{definition}

\begin{remark}
	This is not a conventional monomial ordering because it is not a well-order on all monomials.
	However, we have $F_\pnul A=0$ for some $\pnul>0$ and hence $\gr_i(A)=0$ for all $i\ge \pnul$.
	The conditions (1)(2) make sure that it is a well-order in each degree, and the condition (3) makes sure that it is a well-order on all monomials in degrees less than $\pnul$.
\end{remark}

Assume that $P=k[x_1,\dots,x_n]$ is equipped with a degree-reversed admissible monomial ordering defined above. We denote the leading monomial of $f\in P$ by $\LM(f)$, which is the largest monomial in $f$. The leading term $\LT(f)$ is the leading monomial together with its coefficient in $f$. Recall the following definition of a commutative \Groebner{} basis.

\begin{definition}
	A \Groebner{} basis $G$ of an ideal $I$ of $P$ is a finite generating set of $I$ such that the following two ideals of $P$ are equal:
	$$(\LM(f): f\in I)=(\LM(g): g\in G).$$
\end{definition}

Assume that $G$ is a \Groebner{} basis of the ideal $I$. Let $\sM$ be the set of all monomials in $P$. The definition of \Groebner{} bases indicates that the set
$$\sB=\{m\in \sM : \LM(g)\nmid  m \text{ for all } g\in G\}$$
(modulo $I$) is a basis of $\gr(A)=P/I$ as a vector space. Note that if $m\in \sB$, then all monomials that divide $m$ also belong to $\sB$ by definition.

\begin{definition}
	For any nonzero $f\in \gr(A)=P/I$, it can be uniquely written as
	$$f=c_1m_1+\cdots+c_l m_l \mod I$$
	where $c_i\in k$, $m_i\in \sB$ and $m_1>\cdots>m_l$. The leading monomial of $f$ is defined to be $\LMg(f)=m_1$, while the leading term of $f$ is defined to be $\LTg(f)=c_1m_1$.
\end{definition}

Note that if $m\notin \sB$, then $\LMg(m)<m$.

We choose and fix $X_i\in A$ such that $X_i$ projects to $x_i$ in the associated graded algebra for each $i$. Then for $m=x_1^{r_1}\cdots x_n^{r_n}\in \sB$, the product
$$\tilde m = X_1^{r_1}\cdots X_n^{r_n}\in A$$
projects to $m+I$ since $m+I$ is non-trivial in $\gr(A)=P/I$. Hence the set
$$\tilde \sB=\{\tilde m : m \in \sB\}$$
is a basis of $A$. We order $\tilde \sB$ the same way as we order $\sB$ via the lifting. Note that both $\sB$ and $\tilde \sB$ are well-ordered.

\begin{definition}
	For any $f\in \gr(A)$, we write
	$$f=c_1m_1+\cdots+c_l m_l \mod I$$
	where $c_i\in k$ and $m_i\in \sB$. We define
	$$\tilde f=c_1\tilde m_1+\cdots+c_l\tilde m_l\in A$$
	which is a lift of $f$. Sometimes we also write $\ell(f)=\tilde f$.
\end{definition}

\begin{definition}
	For any nonzero $a\in A$, we can write $a$ in the form of
	$$a=c_1\tilde m_1+\cdots+c_l\tilde m_l$$
	where $c_i\in k$, $m_i\in \sB$ and $m_1>\cdots>m_l$.
	We call $\LMA(a)=\tilde m_1\in \tilde \sB$ the \emph{leading monomial} of $a$, and $\LTA(a)=c_1\tilde m_1$ the \emph{leading term} of $a$. For convenience we also define $\LMB(a)=m_1\in \sB$ and we see that $\LMB=\pr\circ \LM_A$.
\end{definition}

\begin{proposition}
	For any nonzero $a\in A$, we have
	$$\LMB(a)=\pr(\LMA(a))=\LMg(\pr(a))$$
	or equivalently
	$$\LMA=\ell\circ\LMg\circ\pr$$
\end{proposition}

\begin{proof}
	If $$a=c_1\tilde m_1+\cdots+c_l\tilde m_l$$
	where $c_i\in k$, $m_i\in \sB$ and
	$$m_1>\cdots>m_l,$$
	then it is clear that $\pr(a)=\sum\limits_{\substack{1\le i\le l\\|m_i|=|m_1|}}m_i$. Therefore $\LMg(\pr(a))=m_1$.
\end{proof}

The following propositions show that the ordering on $\tilde \sB$ is ``admissible'' meaning that it behaves well with the multiplication. For our convenience, when the context is clear, sometimes we omit ``modulo $I$'' and consider some elements of $P$ as elements of $\gr(A)=P/I$.

\begin{proposition}
	Assume that $m_1, m_2\in \sB$ and $\tilde m_1\tilde m_2\neq 0$. We have
	$$\LMB(\tilde m_1\tilde m_2)\le m_1m_2.$$
	The equality holds if and only if $m_1m_2\in \sB$.
\end{proposition}

\begin{proof}
	By the previous proposition, we have
	$$\LMB(\tilde m_1\tilde m_2)=\LMg(\pr(\tilde m_1\tilde m_2)).$$
	If $m_1m_2$ is trivial in $\gr(A)$, then $|\pr(\tilde m_1\tilde m_2)|>|m_1m_2|$. Therefore
	$$\LMg(\pr(\tilde m_1\tilde m_2))<m_1m_2.$$
	If $m_1m_2$ is nontrivial in $\gr(A)$, we have
	$$\LMg(\pr(\tilde m_1\tilde m_2))=\LMg(m_1m_2)\le m_1m_2.$$
	The last equality holds if and only $m_1m_2\in \sB$.
\end{proof}

This proposition directly implies the following.
\begin{corollary}\label{cor:86d78874}
	Assume $a\in A$, $\LMA(a)=\tilde m$ and $q\in \sB$. If $qm\in \sB$, then $\LMB(\tilde qa)=qm$. Otherwise $\LMB(\tilde qa)<qm$.
\end{corollary}

\section{\Groebner{} bases over Filtered-Commutative Algebras}
Let $A$ be a \FComm{} algebra and we proceed with notations defined in the previous section. We focus on left ideals of $A$ first. We will consider left modules over $A$ in the next section.

\begin{definition}
	A \Groebner{} basis $H$ of a left ideal $J$ of $A$ is a finite generating set of $J$ such that the following two ideals of $P$ are equal:
	$$(\LMB (a): a\in J)=(\LMB (h): h\in H).$$
\end{definition}

In order to give an algorithm that calculates a \Groebner{} basis of $J\subset A$ starting from a given generating set of $J$, we need to define reductions in $A$.

\newcommand{\red}{\mathrm{red}}
\begin{definition}
	Given $a, b\in A$, we write
	$$a=c_1\tilde m_1+\cdots+c_l\tilde m_l$$
	where $c_i\in k$, $m_i\in \sB$ and $m_1>\cdots>m_l$. Let $c\tilde m=\LTA(b)$ for some $c\in k$ and $m\in \sB$. One says that $a$ is \emph{reducible} by $b$ if $m_i$ is divisible by $m$ for some $i$. 
	In this case we choose the smallest $i$ and define the \emph{one-step reduction} of $a$ by $b$ by
	$$\red_1(a, b)=a-\frac{c_i}{c}\widetilde{\left(\frac{m_i}{m}\right)}b.$$
	Note that compared with $a$, $\red_1(a, b)$ replaces $c_i\tilde m_i$ in $a$ with other summands strictly less than $\tilde m_i$ since $$\LTg(\pr(\frac{c_i}{c}\widetilde{\left(\frac{m_i}{m}\right)}b))=\frac{c_i}{c}\frac{m_i}{m}cm=c_im_i.$$
\end{definition}

\begin{definition}
	For $a\in A$ and a finite ordered subset $H\subset A$, we say that $a$ is \emph{reducible} by $H$ if $a$ is reducible by some $h\in H$. In this case we assume that $\tilde m$ is the largest summand of $a$ that is reducible by $H$, and $h\in H$ is the first element of $H$ that reduces $\tilde m$. We define the one-step reduction of $a$ by $H$ by
	$$\red_1(a, H)=\red_1(a, h).$$
	We replace $a$ with $\red_1(a, H)$ and iterate this until $a$ is not reducible by $H$. This will terminate because the monomial ordering is well-ordered on $\tilde \sB$. We call the final outcome the \emph{reduction of $a$ by $H$}, denoted by $\red(a, H)$.
\end{definition}

\begin{remark}\label{rmk:114a6d66}
	By the definition of reductions, if $\red(a,H)=0$, we can write
	$$a=c_1\tilde m_1h_1+\cdots+c_l\tilde m_lh_l$$
	where $c_i\in k$, $m_i\in \sB$ and $h_i\in H$ such that $m_i\LMB(h_i)\in \sB$ and
	$$m_1\LMB(h_1)>\cdots > m_l\LMB(h_l).$$
	We can see $\LMB(a)=m_1\LMB(h_1)$.
\end{remark}

The following is a generalized Buchberger's algorithm that calculates a \Groebner{} basis of an ideal of $A$.

\begin{algorithm}\label{alg:30f8cc72}
	Given a finite ordered generating set $H$ ($0\notin H$) of a left ideal $J$ of $A$, we can expand $H$ to a \Groebner{} basis of $J$ by doing the following
	\begin{enumerate}[(1)]
		\item For $a\in H$ and $g\in G$, let $L$ be the least common multiple
		$$L=\lcm(\LMB(a), \LM(g))$$
		and 
		$$q=\frac{L}{\LMB(a)}.$$
		If $\red(\tilde q a, H)$ is nontrivial, append it to $H$.
		\item For $a, b\in H$, let 
		$$L=\lcm(\LMB(a), \LMB(b))$$
		and
		$$t_1=\frac{L}{\LTg(\pr(a))}, ~t_2=\frac{L}{\LTg(\pr(b))}.$$
		If $\red(\tilde t_1 a-\tilde t_2 b, H)$ 
		is nontrivial, append it to $H$.
		\item Repeat (1)(2) until no more elements can be added to $H$.
	\end{enumerate}
\end{algorithm}

We verify the correctness of the algorithm by the following two propositions.

\begin{proposition}
	Algorithm \ref{alg:30f8cc72} always terminates.
\end{proposition}

\begin{proof}
	It is clear that the new elements added to $H$ belong to $J$. Since each new element is reduced before it is added to $H$, the ideal
	$$(\LMB(h) : h\in H)$$
	of $P$ will be strictly larger when the element is added. This cannot continue infinitely because $P=k[x_1,\dots,x_n]$ is a Noetherian ring.
\end{proof}

\begin{proposition}\label{prop:d2c371dc}
	When Algorithm \ref{alg:30f8cc72} terminates, for any $a\in J$, we have
	$$\LMB (a)\in(\LMB (h): h\in H).$$
\end{proposition}

\begin{proof}
	Since $H$ is a generating set of $J$, it suffices to prove that for all $c_i\in k$, $m_i\in \sB$ and $h_i\in H$, $1\le i\le l$, such that
	$$m_1\LMB(h_1)\ge\cdots \ge m_l\LMB(h_l),$$
	the proposition is true for
	\begin{equation}\label{eq:91004213}
		a=c_1\tilde m_1h_1+\cdots+c_l\tilde m_lh_l.
	\end{equation}

	Without loss of generality, we assume that the leading coefficient of any $h\in H$ is one which implies that $\LTA(h)=\LMA(h)$.

	If $l=0$, the proof is done since this only happens when $a=0$.   

	If $m_1\LMB(h_1)\notin \sB$, then there exists $g\in G$ such that $\LM(g)~|~m_1\LMB(h_1)$.
	Let
	$$L=\lcm(\LMB(h_1), \LM(g))$$
	and 
	$$q=\frac{L}{\LMB(h_1)}.$$
	We have $q|m_1$ and by the termination of Algorithm \ref{alg:30f8cc72}, 
	$$\red(\tilde q h_1, H)=0.$$
	Thus by Remark \ref{rmk:114a6d66} we can write
	$$\tilde qh_1=c_1^\prime\tilde m_1^\prime h_1^\prime+\cdots+c_s^\prime\tilde m_s^\prime h_s^\prime$$
	for some $c_i^\prime\in k$, $m_i^\prime\in \sB$ and $h_i^\prime\in H$ such that
	$$m_1^\prime \LMB(h_1^\prime)>\cdots>m_s^\prime \LMB(h_s^\prime).$$
	We have 
	$$m_j^\prime \LMB(h_j^\prime)\le m_1^\prime \LMB(h_1^\prime)=\LMB(\tilde q h_1)<q\LMB(h_1)$$
	for all $1\le j\le s$.

	Now we consider
	\begin{align}\label{eq:a9d9b43a}
	\begin{split}
		\tilde m_1h_1 &= \Big(\tilde m_1-\widetilde{\left(\frac{m_1}{q}\right)}\tilde q\Big)h_1+\widetilde{\left(\frac{m_1}{q}\right)}(\tilde qh_1)\\
		&= \Big(\tilde m_1-\widetilde{\left(\frac{m_1}{q}\right)}\tilde q\Big)h_1+\sum_{j=1}^s c_j^\prime\widetilde{\left(\frac{m_1}{q}\right)}\tilde m_j^\prime h_j^\prime.
	\end{split}
	\end{align}
	Since $\gr(A)$ is commutative, we have
	$$v\Big(\tilde m_1-\widetilde{\left(\frac{m_1}{q}\right)}\tilde q\Big)>v(\tilde m_1).$$
	Thus
	$$\LMB\Big(\tilde m_1-\widetilde{\left(\frac{m_1}{q}\right)}\tilde q\Big)\cdot \LMB(h_1) < m_1\LMB(h_1).$$
	We also have
	\begin{align*}
		\LMB\Big(\widetilde{\left(\frac{m_1}{q}\right)}\tilde m_j^\prime\Big)\cdot \LMB(h_j^\prime) &\le \frac{m_1}{q}\cdot m_j^\prime \LMB(h_j^\prime)\\
		&< \frac{m_1}{q}\cdot q\LMB(h_1)\\
		&= m_1\LMB(h_1)
	\end{align*}
	for all $1\le j\le s$.
	The above two equalities and (\ref{eq:a9d9b43a}) show that $\tilde m_1h_1$ can be rewritten as a linear combination of elements in the form of $\tilde mh$ where $m\in \sB$, $h\in H$ and $m\LMB(h)<m_1\LMB(h_1)$. We can keep doing the rewriting on (\ref{eq:91004213}) as long as $m_1\LMB(h_1)\notin \sB$ and this has to stop at some point because $|m_1\LMB(h_1)|<2\pnul$ and the monomial ordering is always a well-order on monomials up to a finite degree.

	Now we can always assume that $m_1\LMB(h_1)\in \sB$ on (\ref{eq:91004213}). If $l>1$ and
	$$m_1\LMB(h_1)=m_2\LMB(h_2),$$
	we can rewrite (\ref{eq:91004213}) in a way similar to the previous part of this proof. In fact, let
	$$L=\lcm(\LMB(h_1), \LMB(h_2))$$
	and
	$$q_1=\frac{L}{\LMB(h_1)}, ~q_2=\frac{L}{\LMB(h_2)}.$$
	It is easy to see that there exists $m^\prime \in \sB$ such that $m_1=m^\prime q_1$ and $m_2=m^\prime q_2$.
	By the termination of Algorithm \ref{alg:30f8cc72}, 
	$$\red(\tilde q_1 h_1-\tilde q_2h_2, H)=0.$$
	Thus we can write
	$$\tilde q_1 h_1-\tilde q_2h_2=c_1^\prime\tilde m_1^\prime h_1^\prime+\cdots+c_s^\prime\tilde m_s^\prime h_s^\prime$$
	for some $c_i^\prime\in k$, $m_i^\prime\in \sB$ and $h_i^\prime\in H$ such that
	$$m_1^\prime \LMB(h_1^\prime)>\cdots>m_s^\prime \LMB(h_s^\prime).$$
	We have 
	$$m_j^\prime \LMB(h_j^\prime)\le m_1^\prime \LMB(h_1^\prime)=\LMB(\tilde q_1 h_1-\tilde q_2h_2)<q_1\LMB(h_1)$$
	for all $1\le j\le s$.

	Now we consider
	\begin{align}\label{eq:4dfddfb6}
	\begin{split}
		& \tilde m_1 h_1-\tilde m_2h_2\\
		=& \Big(\tilde m_1-\tilde m^\prime \tilde q_1\Big)h_1-\Big(\tilde m_2-\tilde m^\prime \tilde q_2\Big)h_2+\tilde m^\prime(\tilde q_1h-\tilde q_2h)\\
		=& \Big(\tilde m_1-\tilde m^\prime \tilde q_1\Big)h_1-\Big(\tilde m_2-\tilde m^\prime \tilde q_2\Big)h_2+\sum_{j=1}^s c_j^\prime\tilde m^\prime\tilde m_j^\prime h_j^\prime.
	\end{split}
	\end{align}
	Again, for $i=1,2$ we have
	$$v\Big(\tilde m_i-\tilde m^\prime\tilde q_i\Big)>v(\tilde m_i).$$
	Thus
	$$\LMB\Big(\tilde m_i-\tilde m^\prime\tilde q_i\Big)\cdot \LMB(h_i) < m_i\LMB(h_i)=m_1\LMB(h_1).$$
	We also have
	\begin{align*}
		\LMB\Big(\tilde m^\prime\tilde m_j^\prime\Big)\cdot \LMB(h_j^\prime) &\le \frac{m_1}{q}\cdot m_j^\prime \LMB(h_j^\prime)\\
		&< \tilde m^\prime\cdot q\LMB(h_1)\\
		&= m_1\LMB(h_1)
	\end{align*}
	for all $1\le j\le s$. The above two equalities and (\ref{eq:4dfddfb6}) show that $\tilde m_1h_1$ can be rewritten as a linear combination of $\tilde m_2h_2$ and elements in the form of $\tilde mh$ where $m\in \sB$, $h\in H$ and $m\LMB(h)<m_1\LMB(h_1)$. We can keep doing the rewriting on (\ref{eq:91004213}) as long as $m_1\LMB(h_1)=m_2\LMB(h_2)$ and this will stop at some point.

	Eventually we get $l=0$ or $m_1\LMB(h_1)\in \sB$ and $m_1\LMB(h_1)>m_2\LMB(h_2)$ (when $l>1$).
	In the second case
	$$\LMB(a)=m_1\LMB(h_1)\in (\LMB (h): h\in H).$$
\end{proof}

The \Groebner{} bases over $A$ behave very similarly to commutative \Groebner{} bases. It is quite straightforward to see that $H$ is a \Groebner{} bases of $J\subset A$ if and only if $\red(a, H)=0$ for any $a\in J$, and the set
$$\{\tilde m\in \tilde \sB:\LMB(h)\nmid m\text{ for all } h\in H\}$$
(modulo $J$) is a $k$-basis of the left $A$-module $A/J$.

\section{Modules}
In previous sections, we study the \Groebner{} bases of left ideals of $A$. Now we consider a finitely generated left module $N$ over $A$.
We can find a number $r$ and a left submodule $M$ of $A^r$ such that $N\iso A^r/M$. Therefore we can study left submodules of $A^r$ in order to study finitely generated left modules over $A$.

Consider the truncated graded ring $\bA=\bA_0\oplus \bA_1=A\oplus A^r$, where the multiplication is given by
$$(a_1, x_1)\cdot (a_2, x_2)=(a_1a_2, a_1x_2+a_2x_1)$$
for $a_i\in A$, $x_i\in A^r$, $i=1,2$.
We can see that left submodules of $A^r$ over $A$ are in one-to-one correspondence to left ideals of $\bA$ that are contained in $\bA_1$. We can filter $\bA$ by
$$F_p\bA=F_pA\oplus (F_pA)^r\subset \bA.$$
Then $\bA$ is actually a filtered-graded commutative algebra and
$$\gr(\bA)=\gr(A)\oplus \gr(A)^r\iso k[x_1,\dots,x_n, e_1,\dots,e_r]/(I, e_ie_j: 1\le i\le j\le r)$$
where $\{\tilde e_i\}$ is the $A$-basis of the free $A$-module $\bA_1=A^r$ and $\pr(\tilde e_i)=e_i$.

Since $M\subset \bA_1$, we are only concerned with monomials in the following set
$$\bB=\{me_i: m\in \sM, 1\le i\le r\}.$$
A monomial ordering on those monomials should satisfy
$$m_1e_i<m_2e_i\Longleftrightarrow m_1<m_2 \text{ for all }i.$$
One example of such orderings could be
\begin{equation}\label{eq:808b68a0}
	m_1e_i<m_2e_j\Longleftrightarrow i>j \text{ or } (i=j \text{ and } m_1<m_2).
\end{equation}
There are many other examples but we assume that we have chosen one of them for $\bB$. Note that the set $\tilde\bB=\{\tilde m\tilde e_i:m\in \sM, 1\le i\le r\}$ is a $k$-basis of $A^r$. If $\tilde m\tilde e_i$ is the leading monomial of $x\in A^r$, we write $\LMBB(x)=me_i$.

\begin{definition}
	A \Groebner{} basis $H$ of a left submodule $M\subset A^r$ over $A$ is a \Groebner{} basis of $M\subset \bA_1$ as a left ideal of the filtered-commutative algebra $\bA$. 
\end{definition}

If we expand the definition in terms of $A$, we get the following.

\begin{proposition}
	A \Groebner{} basis $H$ of a left submodule $M\subset A^r$ over $A$ is a finite generating set of $M$ such that the following two ideals of $k[x_1,\dots,x_n, e_1,\dots,e_r]$ are equal:
	$$(\LMBB (x): x\in M)=(\LMBB (h): h\in H).$$
\end{proposition}

Since $M\subset A^r$ can be considered as an ideal of $\bA$, Algorithm \ref{alg:30f8cc72} can be used to compute the \Groebner{} bases of $M$ if we replace $A$ with $\bA$. Here we rewrite the algorithm in terms of $A$ and the auxiliary symbols $e_1,\dots,e_r$.

\begin{algorithm}\label{alg:15793e42}
	Given a finite ordered generating set $H$ ($0\notin H$) of a left submodule $M\subset A^r$ over $A$, we can expand $H$ to a \Groebner{} basis of $M$ by doing the following.
	\begin{enumerate}[(1)]
		\item For $a\in H$ and $g\in G$, let $L$ be the least common multiple
		$$L=\lcm(\LMBB(a), \LM(g))$$
		and 
		$$q=\frac{L}{\LMBB(a)}.$$
		If $\red(\tilde q a, H)$ is nontrivial, append it to $H$.
		\item For $a, b\in H$ such that $\LMBB(a)$ and $\LMBB(b)$ contain the same $e_i$ factor, let 
		$$L=\lcm(\LMBB(a), \LMBB(b))$$
		and
		$$t_1=\frac{L}{\LT_{\gr(\bA)}(\pr(a))}, ~t_2=\frac{L}{\LT_{\gr(\bA)}(\pr(b))}.$$
		If $\red(\tilde t_1 a-\tilde t_2 b, H)$ 
		is nontrivial, append it to $H$.
		\item Repeat (1)(2) until no more elements can be added to $H$.
	\end{enumerate}
\end{algorithm}

\section{Syzygies}
\begin{definition}
	Let $(x_1,\cdots,x_s)$ be a tuple of elements of $A^r$.
	\begin{enumerate}
		\item A \emph{syzygy} of $(x_1,\dots,x_s)$ is a tuple $(a_1,\dots,a_s)\in A^s$ such that $a_1x_1+\cdots+a_sx_s=0$.
		\item The set of all syzygies of $(x_1,\cdots,x_s)$ is call the \emph{(first) syzygy module} of $(x_1,\cdots,x_s)$, denoted by $\Syz(x_1,\cdots,x_s)$.
	\end{enumerate}
\end{definition}

It is obvious that $\Syz(x_1,\cdots,x_s)$ is a left submodule of $A^s$. Our goal in this section is to find a generating set of $\Syz(x_1,\cdots,x_s)$.

\begin{definition}
	Let $H=(h_1,\dots,h_s)$ be an ordered \Groebner{} basis of $M\subset A^r$. For $1\le i\le s$ and $g\in G$, let
	$$L=\lcm(\LMBB(h_i), \LM(g))$$
	and
	$$q=\frac{L}{\LMBB(h_i)}.$$
	We have $\red(\tilde qh_i, H)=0$ and it expands to
	$$\tilde qh_i=a_1h_1+\cdots+a_sh_s$$
	for some $a_l\in A$, $1\le l\le s$. This implies that
	$$(-a_1,\dots,-a_{i-1},\tilde q-a_i,-a_{i+1},\dots,-a_s)$$
	is a syzygy of $H$. We denote this syzygy of $H$ by $\sigma_{i,g}$.
\end{definition}

\begin{definition}
	Let $H=(h_1,\dots,h_s)$ be an ordered \Groebner{} basis of $M\subset A^r$. 
	For $1\le i<j\le s$, let
	$$L=\lcm(\LMBB(h_i), \LMBB(h_j))$$
	and
	$$t_i=\frac{L}{\LT_{\gr(\bA)}(\pr(h_i))}, ~t_j=\frac{L}{\LT_{\gr(\bA)}(\pr(h_j))}.$$
	We have $\red(\tilde t_i h_i-\tilde t_j h_j, H)=0$ and it expands to
	$$\tilde t_i h_i-\tilde t_j h_j=a_1h_1+\cdots+a_sh_s$$
	for some $a_l\in A$, $1\le l\le s$. This implies that
	$$(-a_1,\dots,-a_{i-1},\tilde t_i-a_i,a_{i+1},\dots,-a_{j-1},-\tilde t_j-a_j,-a_{j+1},\dots,-a_s)$$
	is a syzygy of $H$. We denote this syzygy of $H$ by $\sigma_{i,j}$.
\end{definition}

The following proposition finds a generating set of $\Syz(h_1,\dots,h_s)$ for a \Groebner{} basis $(h_1,\dots,h_s)$.

\begin{theorem}
	If $H=(h_1,\dots,h_s)$ is an ordered \Groebner{} basis of some left submodule $M\subset A^r$, then the syzygy module $\Syz(H)$ is generated by the following two sets of elements.
	$$\{\sigma_{i,g}\}_{1\le i\le s, g\in G}, ~\{\sigma_{i,j}\}_{1\le i\le s, 1\le j\le s}.$$
\end{theorem}

\begin{proof}
	This actually follows from the proof of Proposition \ref{prop:d2c371dc} when we replace $A$ with $\bA$ and consider the case $a=0$ in (\ref{eq:91004213}). We can see that the right-hand side of (\ref{eq:91004213}) becomes zero after we rewrite it by linear relations between $H$ that correspond to $m\sigma_{i,g}$ or $m\sigma_{i,j}$ for some $m\in \sM$.
\end{proof}

For any tuple $X=(h_1,\dots,h_t)$ of elements of $A^r$, we can use the generalized Buchberger's algorithm to compute a \Groebner{} basis $H=(h_1,\dots,h_s)$ of the submodule of $A^r$ generated by $h_1,\dots,h_t$ ($s\ge t$). Since all elements of $H$ belong to the submodule, we can find a matrix $Q=\begin{pmatrix}I_t\\Q^\prime\end{pmatrix}\in  M_{s\times t}(A)$ such that
$$H^T=QX^T.$$
Where $I_t$ is the $t\times t$ identity matrix and $(\cdot)^T$ is the transposition of matrices.

\begin{theorem}\label{thm:ef5c591f}
	If $S$ is a matrix such that the row vectors of $S$ generate the syzygy module $\Syz(H)$, then the row vectors of $SQ$ generate $\Syz(X)$.
\end{theorem}
\begin{proof}
	Since $SQX^T=SH^T=0$ we know that the row vectors of $SQ$ belong to $Syz(X)$. On the other hand, assume $vX^T=0$ for some row vector $v\in R^t$. Let 
	$$P=\begin{pmatrix}I_{t} & 0\end{pmatrix}\in M_{t\times s}(A).$$
	Then we have $X^T=PX^T$ and $PQ=I_t$. Thus
	$$0=vX^T=(vPQ)X^T=vPH^T\implies vP\in \Syz(H).$$
	Therefore $vP$ lies in span of row vectors of $S$ and $v=vPQ$ lies in the span of row vectors of $SQ$.
\end{proof}

Combining the two theorems above solves the computation of $\Syz(X)$ for any tuple $X$.

\section{Ext Groups}
In this section, we further assume that $A=\oplus_{i\ge 0}A_i$ is a connected graded algebra, which means that $A_0=k$. The degree of a (homogeneous) element $a\in A$ is denoted by $\deg(a)$. We require the filtration $F_\bullet A$ to be homogeneous. Thus the associated graded algebra $\gr(A)$ has two gradings. To avoid confusion, the degree function on $\gr(A)$ induced by the grading of $A$ is also denoted by $\deg(\cdot)$ while the degree function induced by the filtration is denoted by $|\cdot|$ as before.

Let $N=A^r/M$ be a graded left $A$-module. Here $A^r=A\{\tilde v_i:1\le i\le r\}$ is a direct sum of shifts of $A$ where the degrees of basis $\deg(\tilde v_i)$ are integers. The goal of this section is to compute the Ext groups $\Ext^*_{A}(N,k)$ by constructing a minimal free resolution of $N$. Most of the algorithms are adapted from the theory of commutative \Groebner{} bases (See \cite{KR05}). The only differences are the steps that construct noncommutative \Groebner{} bases.

First we have to make sure that the number $r$ in the representation $N=A^r/M$ is as small as possible. Let $A_+=\oplus_{i>0}A_i$ be the ``graded maximal ideal" of $A$. By Nakayama's lemma, we see that a minimal generating set of the graded left module $N$ of $A$ corresponds to a basis of the $k$-vector space $N/A_+N$. Hence we can use the following algorithm to minimize $r$.

\begin{algorithm}\label{alg:798e215b}
	Let $N$ be a graded left submodule over $A$ with the following presentation
	$$N=A\{\tilde v_1,\dots,\tilde v_r\}/(Ax_1+\cdots+Ax_l)$$
	where $x_i$ are $A$-linear combinations of $\tilde v_j$. We can minimize $r$ in the presentation by the following steps.
	\begin{enumerate}
		\item For $1\le i\le l$, let $y_i$ be the sum of all monomials in $x_i$ that equals $\tilde v_j$ for some $j$. In other words, we remove monomials in $A_+\{\tilde v_1,\dots,\tilde v_r\}$ from summands of $x_i$. Then we have
		$$N/A_+N\iso k\{\tilde v_1,\dots,\tilde v_r\}/(ky_1+\cdots+ky_l).$$
		\item Assume that $\dim_k(N/A_+N)=r^\prime$ for some $r^\prime\le r$. Perform row reductions on $y_i$ to find a subset of $\{\tilde v_1,\dots,\tilde v_r\}$ that generates $N/A_+N$. Without loss of generality we assume that $\{\tilde v_1,\dots,\tilde v_{r^\prime}\}$ generates $N/A_+N$ and we can write $\tilde v_j$ as a $k$-linear combination of $\tilde v_1,\dots,\tilde v_{r^\prime}$ for $j>r^\prime$ in $N/A_+N$.
		\item Perform the exact same row reductions on $x_i$ instead of $y_i$. Then in $N$ we can write $\tilde v_j$ as a $k$-linear combination of $\tilde v_1,\dots,\tilde v_{r^\prime}$ and other monomials in $A_+\{\tilde v_1,\dots,\tilde v_r\}$ for $j>r^\prime$. By iterating this process $\tilde v_j$ can be further rewritten as an $A$-linear combination of $\tilde v_1,\dots,\tilde v_{r^\prime}$. Thus we can replace $x_i$ with
		$$x_i^\prime\in Av_1+\cdots+Av_{r^\prime}$$
		such that
		$$N\iso A\{\tilde v_1,\dots,\tilde v_{r^\prime}\}/(Ax_1^\prime+\cdots+Ax_l^\prime)$$
		where $\{\tilde v_1,\dots,\tilde v_{r^\prime}\}$ is a minimal generating set of $N$.
	\end{enumerate}
\end{algorithm}

Next we need the following simple algorithm which computes a minimal generating set of of a submodule $M\subset A^r$.

\begin{algorithm} \label{alg:76dc86bf}
	Let $M$ be a graded left submodule of $A^r$ generated by $x_1,\dots,x_l$. We can find a subset of $\{x_1,\dots,x_l\}$ that is a minimal generating set of $M$ by the following steps.
	\begin{enumerate}
		\item Order $x_1,\dots,x_l$ by degrees such that
		$$\deg(x_1)\le\cdots\le\deg(x_l).$$
		\item Compute the \Groebner{} basis $H_i$ of the submodule generated by $x_1,\dots,x_i$ for $0\le i\le l$ ($H_0=\emptyset$). We can apply Algorithm \ref{alg:15793e42} on $H_{i-1}\cup \{\red(x_{i}, H_{i-1})\}$ to obtain $H_i$.
		\item When $\red(x_{i}, H_{i-1})=0$, mark $x_i$ as redundant for $1\le i\le l$.
		\item Remove the redundant elements from $x_1,\dots,x_l$. The remaining elements form a minimal generating set of of $M$.
	\end{enumerate}
\end{algorithm}

Now we have all the ingredients to construct a minimal resolution of $N=A^r/M$.

\begin{algorithm}\label{alg:353e422a}
	Let $N=A^r/M$ be a graded left $A$-module where $M\subset A^r$ is generated by $x_1,\dots,x_l$. We construct the first $s+1$ terms $F_0,\dots,F_s$ of a free resolution
	$$\cdots\to F_s\fto{d_s}\cdots\fto{d_3} F_2\fto{d_2} F_1\fto{d_1} F_0=A^r\fto{\epsilon} N$$
	by the following steps.
	\begin{enumerate}
		\item Apply Algorithm \ref{alg:798e215b} to minimize $r$ in the presentation of $N$.
		\item Let $d_0=\epsilon: F_0=A^r\to N$ be the quotient map. Apply Algorithm \ref{alg:76dc86bf} on $M=\ker(d_0)$ to obtain a minimal generating set $\{x_{11},\dots,x_{1l_1}\}$ of $M$.
		\item For $i=1,2,\dots,s$, assume that $\{x_{i1},\dots,x_{il_i}\}$ is a minimal generating set of $\ker(d_{i-1})$. Define 
		$$F_{i}=A\{v_{i1},\dots,v_{il_i}\}, ~\deg(v_{ij})=\deg(x_{ij})$$
		and
		$$d_{i}(v_{ij})=x_{ij}, \text{ for } 1\le j\le l_i.$$
		Apply Theorem \ref{thm:ef5c591f} to compute a minimal generating set of $\ker(d_{i})\iso\Syz(x_{i1},\dots,x_{il_i})$.
	\end{enumerate}
\end{algorithm}

By executing the algorithm we can find that 
$$\Ext_A^{i}(N,k)\iso\begin{cases}
	k^r & i=0,\\
	k^{l_i} & 1\le i\le s.
\end{cases}$$

We call this algorithm a vertical method for computing a minimal resolution. We can write a horizontal version of this algorithm which executes this algorithm degree by degree (from the grading of $A$). We leave it for the reader to write up this version. We refer the reader to \cite[Theorem 4.8.16]{KR05} of a commutative version of a horizontal method which incorporates the elimination algorithm as an optimization and it can be directly transferred to our case.

\section{The Steenrod Algebra}
In the rest of this paper, we would like to apply our noncommutative \Groebner{} bases to the Steenrod algebra at the prime 2 and compute the Ext groups over it. In this section we recall some facts about the Steenrod algebra.

Let $\sA$ be the Steenrod algebra at the prime 2. It is a Hopf algebra generated by symbols $Sq^n$, $n\ge 1$ in degree $n$ with relations given by
$$Sq^iSq^j=\sum_{k=0}^{[i/2]}\binom{j-k-1}{i-2k}Sq^{i+j-k}Sq^k \hspace{0.5cm} (Sq^0=1)$$
for all $i,j>0$ such that $i<2j$. They are call the Adem relations. The coproduct is given by
$$\psi(Sq^n)=\sum_{i=0}^n Sq^i\otimes Sq^{n-i}.$$

In order to define a suitable filtration for $\sA$, first we consider the dual of $\sA$. By the work of Milnor \cite{Milnor58}, the dual of the Steenrod algebra $\sA_*$ can be characterized by
$$\sA_*= \bF_2[\xi_1,\xi_2,\dots], \hspace{0.5cm}\deg(\xi_i)=2^i-1$$
with coproduct given by
$$\psi(\xi_n)=\sum_{i=0}^n \xi_{n-i}^{2^i}\otimes \xi_i. \hspace{0.5cm} (\xi_0=1)$$
The dual basis of the monomial basis of $\sA_*$ is denoted by $\{P(R)\}$ where $R=(r_1,r_2,\dots)$ range over all sequences of non-negative integers which are almost all zero, and $P(R)$ is dual to
$$\xi(R)=\xi_1^{r_1}\xi_2^{r_2}\cdots.$$
Milnor gives a product formula which computes a product $P(R_1)\cdot P(R_2)$ as a linear combination of $P(R)$ again.

We define a weight function $w$ on $\sA$ by
$$w(P(r_1,r_2,\dots))=\sum_{k,i}(2k-1)a_{k,i}$$
where $r_k=\sum_i a_{k,i}2^i$ is the 2-adic expansion. We define a decreasing filtration on $\sA$ by
$$F_p\sA=k\{P(R): w(P(R))\ge p\}.$$
It is not hard to check via the coproduct formula in the dual that $F_\bullet \sA$ satisfies
$$F_p\sA\cdot F_q\sA\subset F_{p+q}\sA$$
and in each degree $n$, we have $F_p\sA_n=0$ when $p$ is large enough.

\begin{proposition}
	The associated graded algebra of the Steenrod algebra is isomorphic to an exterior algebra
	$$\gr(\sA)\iso E[P^i_j:i\ge 0, j>0]$$
	where the symbol $P^i_j$ is the projection of
	$$\tilde P^i_j=P(0,\dots,0,\overset{j}{2^i},0,\dots)$$
	which is dual to $\xi_j^{2^i}$. The gradings of $\gr(\sA)$ are given by
	$$\deg(P^i_j)=2^i(2^j-1), ~|P^i_j|=2j-1.$$
\end{proposition}

The reader can find a proof of this theorem from the proof of \cite[Theorem 3.2.2]{Ravenel86}.

\begin{remark}
	This associated graded algebra here is actually the same as the associated homogeneous Koszul algebra (defined by Priddy \cite{Priddy70}) of May's associated graded algebra of the Steenrod algebra (see \cite{May64}).
\end{remark}

\begin{remark}
	For prime $p>2$, there is a similar filtration for the Steenrod algebra at the prime $p$ such that the associated graded algebra is a product of exterior algebras and truncated polynomial algebras of height $p$ (see the remarks before \cite[Lemma 3.2.4]{Ravenel86}). Therefore we can also apply our algorithms to the Steenrod algebras at odd primes. However, the author has not implemented these algorithms into computer programs yet.
\end{remark}

Although $\gr(\sA)$ is not a finitely generated commutative algebra, the truncated graded algebra $\sA_{\le n}$ is finitely generated and so is $\gr(\sA_{\le n})$. When we want to compute a minimal resolution of a left $\sA$-module up to degree $n$, it suffices to use $\sA_{\le n}$ instead of $\sA$. 

For our convenience, in the rest of this section, we implicitly consider everything truncated to some degree so that \Groebner{} bases are still well-defined.

When $A=\sA$, consider
$$\gr(\sA)= k[P^i_j:i\ge 0, j>0]/I$$
where the ideal is generated by the set of squares 
$$G=\{(P^i_j)^2:i\ge 0, j>0\}.$$
It is obvious that $G$ is a \Groebner{} basis of $I$. We order the generators $P^i_j$ by degree. In other words, we have
$$P^i_j<P^s_t\Longleftrightarrow \deg(P^i_j)<\deg(P^s_t).$$
This is well defined since all $P^i_j$ have different degrees. We order monomials of $P^i_j$ such that
\begin{align*}
	m<m^\prime\Longleftrightarrow & |m|<|m^\prime| \text{ or } \\
	& |m|=|m^\prime| \text{ and } m \text{ is \emph{greater} than } m^\prime \text{ lexicographically}.
\end{align*}
The author uses this particular monomial ordering because the corresponding computer program is faster than other versions the author has tried so far. For left submodules of $\sA^r$ we use (\ref{eq:808b68a0}) as our monomial ordering.

The set
$$\sB=\{\text{square-free monomials in variables } P^i_j\}$$
is a $\bF_2$-basis of $\gr(\sA)$ while the set
$$\tilde\sB=\{\tilde P^{i_1}_{j_1}\cdots \tilde P^{i_l}_{j_l}:\tilde P^{i_1}_{j_1}<\cdots< \tilde P^{i_l}_{j_l}, l\ge 0\}$$
is a $\bF_2$-basis of $\sA$. The basis $\tilde \sB$ is similar to the ``$P^s_t$ basis'' in Monks \cite{Monks98} with a different orderings of generators.

\section{Computing the Ext Groups over the Steenrod Algebra}
In the previous section, we have shown that the truncated graded algebra $\sA_{\le n}$ is a finitely generated filtered-commutative algebra. 
Therefore we can apply the noncommutative \Groebner{} bases to compute the Ext groups $\Ext_\sA^{*,*}(N, k)$ for a finitely generated graded left $\sA$-module $N$.

The author has implemented the algorithms into computer programs to compute $\Ext_\sA^{*,*}(N, k)$ for various $N$ up to some degree. The program code can be found on Github \cite{LinGH}. In addition to Algorithm \ref{alg:353e422a}, the computer programs incorporate the following optimizations in commutative algebras which can be easily adapted to our noncommutative case.
\begin{enumerate}
	\item Use Buchberger triple to reduce the number of steps needed to compute a \Groebner{} basis. (See \cite[Tutorial 25]{KR08})
	\item Build the resolution degree by degree and use eliminations. (See \cite[Theorem 4.8.16]{KR05})
\end{enumerate}

Since $\sA$ is a Hopf algebra, $\Ext_\sA^{*,*}(\bF_2,\bF_2)$ is actually a commutative algebra over $\bF_2$. After we obtain a resolution by the programs above, we can use methods described in \cite{Bruner89, Bruner93} to compute the products in $\Ext_\sA^{*,*}(\bF_2,\bF_2)$ and module structures of $\Ext_\sA^{*,*}(N,\bF_2)$ over $\Ext_\sA^{*,*}(\bF_2,\bF_2)$ for $A$-modules $N$. However, we always use \Groebner{} bases to do linear algebra other than enumerating $\bF_2$-bases of free $\sA$-modules. The elimination algorithm mentioned above is especially efficient for building maps between chain complexes, which is used to compute products in Ext.

We have computed the cohomology of the Steenrod algebra $\Ext_\sA^{*,*}(\bF_2,\bF_2)$ with outcomes described in Theorem \ref{thm:f088d56}. The computed range is very large and contains $h_7^2$ and $h_8$. The reader can find a graphical diagram of $\Ext_{\sA}^{*,*}(\bF_2,\bF_2)$ on the webpage \cite{LinGraph}.

\makebibliography

\hspace{1cm}\vspace{5pt}

{\scshape Weinan Lin, School of Mathematical Sciences, Peking University, 5 Yi He Yuan Road, Haidian District, Beijing, 100871, China.}

\emph{Email address}: \url{lwnpku@math.pku.edu.cn}

\end{document}